\documentclass[a4paper,reqno,draft]{amsart}

\usepackage[T2A]{fontenc}
\usepackage[cp1251]{inputenc}
\usepackage{amssymb,amscd,amsthm,amsaddr}

\theoremstyle{plain}
\newtheorem{lem}{Lemma}

\theoremstyle{plain}
\newtheorem{thm}{Theorem}

\theoremstyle{definition}
\newtheorem{rem}{Remark}

\newcounter{sv}
\theoremstyle{plain}
\newtheorem*{ph}{\addtocounter{sv}{1}$\mathbf{\Phi}$\arabic{sv}}

\makeatletter
\renewcommand{\email}[2][]{%
	\ifx\emails\@empty\relax\else{\g@addto@macro\emails{,\space}}\fi%
	\@ifnotempty{#1}{\g@addto@macro\emails{#2\space}}%
	\g@addto@macro\emails{(\textrm{#1})}%
}
\makeatother

\begin{document}
\author[A.\,Yu.\,Ulitskaya]{A.\,Yu.\,Ulitskaya}
\address{Chebyshev Laboratory, St. Petersburg State University,\\14th Line V.O., 29B, Saint Petersburg 199178 Russia}
\email[A.\,Yu.\,Ulitskaya]{baguadadao@gmail.com}

\title[Fourier analysis in spaces of shifts]{Fourier analysis in spaces of shifts}

\begin{abstract}
In this paper, we develop a continual analog of decomposition over orthogonal bases in spaces generated by equidistant shifts of a single function. By doing so, we obtain an explicit expression for best approximation by spaces of shifts in $L_2(\mathbb{R})$. The result is formulated in terms of classical Fourier transform and tends to have various applications in approximation by spaces of shifts and, in particular, in spline approximation.

\end{abstract}
	
\maketitle

\section{Notation}
In what follows, $\mathbb{R}$, $\mathbb{Z}$, $\mathbb{Z}_+$, $\mathbb{N}$ are the sets of real, integer, nonnegative integer, and natural numbers, respectively. Unless otherwise follows from the context, all the functional spaces under consideration can be real or complex.

If $p\in[1,+\infty)$, $L_p(\mathbb{R})$ is the space of measurable, $p$-integrable on $\mathbb{R}$ functions; $L_\infty(\mathbb{R})$ is the space of essentially bounded functions on $\mathbb{R}$. The norms in these spaces are defined by
$$
\|f\|_{L_p(\mathbb{R})}=\left(\int\limits_\mathbb{R}|f|^p\right)^{1/p},\quad
\|f\|_{L_\infty(\mathbb{R})}=\underset{\mathbb{R}}{\mathrm{vrai\,sup\,}}|f|,
$$
respectively.
The spaces $L_p[\,a,b]$ are defined similarly. Furthermore, for $p\in[1,+\infty)$, let $\ell_p(\mathbb{Z})$ be
the space of two-sided sequences~${a=\{a_k\}_{k\in\mathbb{Z}}}$ for which $\|a\|_{\ell_p(\mathbb{Z})}=\left(\sum\limits_{k\in\mathbb{Z}}|a_k|^p\right)^{1/p}<+\infty$. 
  
The space of compactly supported infinitely differentiable functions on $\mathbb{R}$ is denoted by $C^{(\infty)}_0(\mathbb{R})$.

The symbol $\langle\cdot,\cdot\rangle_\mathcal{H}$ denotes the inner product in the Hilbert space~$\mathcal{H}$;
$$
E(f,\mathfrak{N})_p=\inf\limits_{T\in\mathfrak{N}}\|f-T\|_p
$$
is best approximation of $f$ in $L_p$ by the set $\mathfrak{N}\subset L_p$.

The Fourier transform of the function $f\in L_1(\mathbb{R})$ is defined by the formula
$$
\widehat{f}(y)=\frac{1}{2\pi}\int\limits_\mathbb{R}f(x)e^{-ixy}\,dx.
$$
The definition of the Fourier transform of $f\in L_2(\mathbb{R})$ is extended from $L_1(\mathbb{R})\cap L_2(\mathbb{R})$ standardly.

For $\sigma>0$, $\mu\in\mathbb{Z}_+$, denote by $\mathbf{S}_{\sigma,\mu}$ the space of splines of degree~$m$ and defect~$1$ with knots at the points $\dfrac{j\pi}{\sigma}$, $j\in\mathbb{Z}$.

\section{Introduction}
Spaces generated by shifts of a single function play a significant role in approximation theory, wavelets theory, and applications.
Approximative properties of such spaces are of wide interest; see, for example~\cite{skopina1,ron2,ol4}.

Moreover, spaces of shifts often become optimal approximating subspaces in the sense of widths. This means that constants in estimates for approximation of certain classes of functions by spaces of shifts cannot be decreased by passing to another approximating subspace of the same dimension (or, in the infinite dimensional case, average dimension).
The most classical result on this subject is probably optimality of spline spaces for Sobolev classes of functions (see, for example,~\cite{korn,sun}). In the general situation, extremality results for approximation by spaces of shifts were obtained by Vinogradov (\cite{ol1,ol2,ol3}) in periodic and non-periodic spaces $C$ and $L_1$.

In~\cite{we} Vinogradov and Ulitskaya studied the following problem of $L_2$-approximation by spaces of shifts. For $n\in\mathbb{N}$ and periodic function $B\in L_1$ consider the space~$\mathcal{S}_{B,n}$ of functions~$s$ that are defined on $\mathbb{R}$ and representable in the form
$$
s(x)=\sum\limits_{j=0}^{2n-1}\beta_jB\left(x-\frac{j\pi}{n}\right).
$$
It is easy to show that the space~$\mathcal{S}_{B,n}$ coincides with the linear span of the functions
$$
\Phi_{B,l}(x)=\Phi_{B,n,l}(x)=\frac{1}{2n}\sum\limits_{j=0}^{2n-1}e^{\frac{ilj\pi}{n}}B\left(x-\frac{j\pi}{n}\right),\quad l\in[1-n:n].
$$
If the set $\{B\left(\cdot-\frac{j\pi}{n}\right)\}_{j=1-n}^{n}$ is linearly independent, the functions $\Phi_{B,l}$ form an orthogonal basis in the space~$\mathcal{S}_{B,n}$. For $m<n$, denote by $\mathcal{S}^\times_{B,n,m}$ the linear span of the set $\{\Phi_{B,l}\}_{l=1-m}^{m-1}$.

We find optimality criteria for approximation of Sobolev classes of periodic functions by the spaces~$\mathcal{S}_{B,n}$ and~$\mathcal{S}^\times_{B,n,m}$ in $L_2$. These criteria are given in terms of Fourier coefficients of the function~$B$. The proof is based on Fourier series decomposition over orthogonal bases in spaces of shifts formed by the functions $\Phi_{B,l}$. This decomposition allows us to find an explicit expression for best approximation by $\mathcal{S}_{B,n}$ and~$\mathcal{S}^\times_{B,n,m}$. Later this technique was used in~\cite{me} to obtain a criteria of optimality of spaces under consideration for classes of periodic convolutions.


Similar approximation problems in the space $L_2(\mathbb{R})$ force us to develop an analog of this method in continual case. In this paper, we elaborate the Fourier transform method in spaces of shifts and obtain an explicit expression for best mean-square approximation by these spaces in terms of classical Fourier transform.

\section{Main results}
Let $\sigma>0$, $B\in L_2(\mathbb{R})$ and consider the functions
\begin{align}
	\label{1.0}
	\Phi_{B,\sigma}(x,y)&=\frac{1}{2\sigma}\sum\limits_{j\in\mathbb{Z}}B\left(x-\frac{j\pi}{\sigma}\right)e^{i\tfrac{j\pi}{\sigma}y},\\
	\label{1.1}
	D_{B,\sigma}(y)&=\sum\limits_{\nu\in\mathbb{Z}}\left|\widehat{B}(y+2\nu\sigma)\right|^2.
\end{align}
The definition of $\Phi_{B,\sigma}$ is correct since for a.e. $x\in\mathbb{R}$ the series on the right-hand side of~\eqref{1.0} converges in $L_2[\,-\sigma,\sigma]$.

Let us mention some properties of the functions~$\Phi_{B,\sigma}$. Note that they are quite similar to the properties of the Zak transform (see, for example,~\cite[\S8.2]{grech}).

\begin{ph}
	$\Phi_{B,\sigma}\in L_2\left(\left[\,0,\frac{\pi}{\sigma}\right]\times[\,-\sigma,\sigma]\right)$	and
	$$
	\int\limits_{0}^{\frac{\pi}{\sigma}}
	\int\limits_{-\sigma}^\sigma\left|\Phi_{B,\sigma}(x,y)\right|^2dy\,dx=
	\frac{1}{2\sigma}\|B\|^2_{L_2(\mathbb{R})}.
	$$
\end{ph}
\begin{proof}
	Since $\left\{B\left(x-\frac{j\pi}{\sigma}\right)\right\}_{j\in\mathbb{Z}}
	\in\ell_2(\mathbb{Z})$ for a.e.~$x\in\mathbb{R}$, we get 
	$\Phi_{B,\sigma}(x,\cdot)\in L_2[-\sigma,\sigma]$ for a.e.~$x$. Furthermore, by Parseval's equality, we have
	$$
	\int\limits_{-\sigma}^\sigma\left|\Phi_{B,\sigma}(x,y)\right|^2dy=
	\frac{1}{2\sigma}\sum\limits_{j\in\mathbb{Z}}\left|B\left(x-\frac{j\pi}{\sigma}\right)\right|^2.
	$$
This yields
	$$
	\int\limits_{0}^{\frac{\pi}{\sigma}}
	\int\limits_{-\sigma}^\sigma\left|\Phi_{B,\sigma}(x,y)\right|^2dy\,dx=
	\frac{1}{2\sigma}\int\limits_{0}^{\frac{\pi}{\sigma}}
	\sum\limits_{j\in\mathbb{Z}}\left|B\left(x-\frac{j\pi}{\sigma}\right)\right|^2dx=
	\frac{1}{2\sigma}\|B\|^2_{L_2(\mathbb{R})}.
	$$
\end{proof}

\begin{ph}
	The following relations hold:
	\begin{gather*}
		\Phi_{B,\sigma}(x,y+2\sigma)=\Phi_{B,\sigma}(x,y),\quad\Phi_{B,\sigma}\left(x+\frac{\pi}{\sigma},y\right)=
		e^{i\frac{\pi}{\sigma}y}\Phi_{B,\sigma}(x,y),\\ \overline{\Phi_{B,\sigma}(x,y)}=\Phi_{B,\sigma}(x,-y).
	\end{gather*}	
\end{ph}

This property is clear. Moreover, it implies that the function~$\Phi_{B,\sigma}$ is completely defined by its values on $\left[\,0,\frac{\pi}{\sigma}\right]\times[\,-\sigma,\sigma]$.

\begin{ph}
	In $L_2\left(\left[\,0,\frac{\pi}{\sigma}\right]\times[\,-\sigma,\sigma]\right)$, we have
	\begin{equation}
		\label{1.2}
		\Phi_{B,\sigma}(x,y)=\sum\limits_{\nu\in\mathbb{Z}}
		\widehat{B}(y+2\nu\sigma)e^{i(y+2\nu\sigma)x}.
	\end{equation}
\end{ph}	
\begin{proof}
	First we point out that one can obtain the equality
	\begin{equation}
		\label{1.22}
		\int\limits_{-\sigma}^\sigma\int\limits_0^{\frac{\pi}{\sigma}}\left|\sum\limits_{\nu\in\mathbb{Z}}
		\widehat{B}(y+2\nu\sigma)e^{i(y+2\nu\sigma)x}\right|^2dx\,dy=
		\frac{1}{2\sigma}\|B\|^2_{L_2(\mathbb{R})},
	\end{equation}
using the same argument as in the proof of~$\Phi1$.
	Indeed, since $\left\{\widehat{B}(y+2\nu\sigma)\right\}_{\nu\in\mathbb{Z}}
	\in\ell_2(\mathbb{Z})$ for a.e.~$y\in\mathbb{R}$, by Parseval's equality, we have for a.e.~$y\in\mathbb{R}$
	\begin{gather*}
		\int\limits_0^{\frac{\pi}{\sigma}}\left|\sum\limits_{\nu\in\mathbb{Z}}
		\widehat{B}(y+2\nu\sigma)e^{i(y+2\nu\sigma)x}\right|^2dx=
		\int\limits_0^{\frac{\pi}{\sigma}}\left|\sum\limits_{\nu\in\mathbb{Z}}
		\widehat{B}(y+2\nu\sigma)e^{2i\nu\sigma x}\right|^2dx=\\
		=\frac{\pi}{\sigma}\sum\limits_{\nu\in\mathbb{Z}}
		\left|\widehat{B}(y+2\nu\sigma)\right|^2.
	\end{gather*}
	Therefore,
	\begin{gather*}
		\int\limits_{-\sigma}^\sigma\int\limits_0^{\frac{\pi}{\sigma}}\left|\sum\limits_{\nu\in\mathbb{Z}}
		\widehat{B}(y+2\nu\sigma)e^{i(y+2\nu\sigma)x}\right|^2dx\,dy=
		\frac{\pi}{\sigma}\int\limits_{-\sigma}^\sigma
		\sum\limits_{\nu\in\mathbb{Z}}\left|\widehat{B}(y+2\nu\sigma)\right|^2\,dy=\\
		=\frac{\pi}{\sigma}\|\widehat{B}\|^2_{L_2(\mathbb{R})}=
		\frac{1}{2\sigma}\|B\|^2_{L_2(\mathbb{R})}.
	\end{gather*}
	
Let $B_n\in C_0^{(\infty)}(\mathbb{R})$, $B_n\to B$ in $L_2(\mathbb{R})$. For a.e.~$x,y\in\mathbb{R}$ put
	$$
	g_n(t)=e^{iyt}B_n(x-t),\quad t\in\mathbb{R}.
	$$
	Then $g_n\in C_0^{(\infty)}(\mathbb{R})$,
	$$
	\widehat{g_n}(\xi)=e^{i(y-\xi)x}\widehat{B_n}(y-\xi),\quad \xi\in\mathbb{R},
	$$
	and the Poisson summation formula yields
	$$
	\Phi_{B_n,\sigma}(x,y)=\frac{1}{2\sigma}\sum\limits_{j\in\mathbb{Z}}
	g_n\left(j\frac{\pi}{\sigma}\right)=\sum\limits_{j\in\mathbb{Z}}
	\widehat{g_n}(2j\sigma)=
	\sum\limits_{j\in\mathbb{Z}}\widehat{B_n}(y+2j\sigma)e^{i(y+2j\sigma)x}.
	$$
	Using property~$\Phi$1 and identity~\eqref{1.22}, by passing to the limit as $n\to\infty$, we get the required assertion.
\end{proof}

Recall that a set $\{f_j\}_{j\in\mathbb{Z}}$ of elements of a Hilbert space $\mathcal{H}$ is called \emph{a Riesz system} with constants $A, B>0$ if for any $\beta\in\ell_2(\mathbb{Z})$ the series $\sum\limits_{j\in\mathbb{Z}}\beta_jf_j$ converges in $\mathcal{H}$ and
\begin{equation}\label{1.222}
	A\|\beta\|^2_{\ell_2(\mathbb{Z})}\leqslant\left\|\sum\limits_{j\in\mathbb{Z}}\beta_jf_j\right\|^2_\mathcal{H}\leqslant B\|\beta\|^2_{\ell_2(\mathbb{Z})}.
\end{equation}
If we assume that only the right-hand inequality in~\eqref{1.222} holds, then $\{f_j\}_{j\in\mathbb{Z}}$ is said to be \emph{a Bessel system}.

\begin{ph}
	Let $\left\{B\left(\cdot-\frac{j\pi}{\sigma}\right)\right\}_{j\in\mathbb{Z}}$ be a Bessel system in $L_2(\mathbb{R})$. Then the following equality holds
	\begin{equation}
		\label{1.3}
		\int\limits_{\mathbb{R}}B(t)\overline{\Phi_{B,\sigma}(t,y)}\,dt=
		2\pi D_{B,\sigma}(y)
	\end{equation}	
	(convergence of the integral on the left-hand side of~\eqref{1.3} is interpreted in  $L_2[\,-\sigma,\sigma]$).
\end{ph}

\begin{proof}
	Indeed, for $n\in\mathbb{N}$, by formulas~\eqref{1.2}, \eqref{1.1},	 we have
	\begin{gather*}	
		\int\limits_{-\sigma}^\sigma\left|\,
		\int\limits_{-n}^n B(t)\overline{\Phi_{B,\sigma}(t,y)}\,dt
		-2\pi D_{B,\sigma}(y)\right|^2dy=\\
		=\int\limits_{-\sigma}^\sigma\left|\,
		\int\limits_{-n}^n B(t)\overline{\sum\limits_{\nu\in\mathbb{Z}}
			\widehat{B}(y+2\nu\sigma)e^{i(y+2\nu\sigma)t}}\,dt
		-2\pi \sum\limits_{\nu\in\mathbb{Z}}\left|\widehat{B}(y+2\nu\sigma)\right|^2\right|^2dy=\\
		=4\pi^2\int\limits_{-\sigma}^\sigma
		\left|\sum\limits_{\nu\in\mathbb{Z}}
		\overline{\widehat{B}(y+2\nu\sigma)}
		\left(\frac{1}{2\pi}\int\limits_{-n}^n B(t)e^{-i(y+2\nu\sigma)t}\,dt
		-\widehat{B}(y+2\nu\sigma)\right)\right|^2dy.
	\end{gather*}
	Termwise inner multiplication in the second equation is valid since series~\eqref{1.2} converges in $L_2[-n,n]$ with respect to $x$. This implies by Cauchy-Schwarz inequality
	\begin{gather*}
		\int\limits_{-\sigma}^\sigma\left|\,
		\int\limits_{-n}^n B(t)\overline{\Phi_{B,\sigma}(t,y)}\,dt
		-2\pi D_{B,\sigma}(y)\right|^2dy\leqslant\\
		\leqslant4\pi^2
		\int\limits_{-\sigma}^\sigma \left(\sum\limits_{\nu\in\mathbb{Z}}
		|\widehat{B}(y+2\nu\sigma)|^2\right)
		\left(\sum\limits_{\nu\in\mathbb{Z}}\left|\frac{1}{2\pi}
		\int\limits_{-n}^n B(t)e^{-i(y+2\nu\sigma)t}\,dt
		-\widehat{B}(y+2\nu\sigma)\right|^2\right)\,dy\leqslant\\
		\leqslant4\pi^2\|D_{B,\sigma}\|_{L_{\infty}[\,-\sigma,\sigma]}
		\int\limits_{-\sigma}^\sigma
		\sum\limits_{\nu\in\mathbb{Z}}\left|\frac{1}{2\pi}
		\int\limits_{-n}^n B(t)e^{-i(y+2\nu\sigma)t}\,dt
		-\widehat{B}(y+2\nu\sigma)\right|^2dy=\\
		=4\pi^2\|D_{B,\sigma}\|_{L_{\infty}[\,-\sigma,\sigma]}
		\int\limits_{\mathbb{R}}
		\left|\frac{1}{2\pi}\int\limits_{-n}^n B(t)e^{-iyt}\,dt-\widehat{B}(y)\right|^2dy\xrightarrow[n\to\infty]{}0.
	\end{gather*}
	Finiteness of the value $\|D_{B,\sigma}\|_{L_{\infty}[\,-\sigma,\sigma]}$ is ensured by the fact that the functions $\left\{B\left(\cdot-\frac{j\pi}{\sigma}\right)\right\}_{j\in\mathbb{Z}}$ form a Bessel system (see, for example,~\cite[Remark~1.1.7]{skopina}).	
\end{proof}	

Suppose that $\left\{B\left(\cdot-\frac{j\pi}{\sigma}\right)\right\}_{j\in\mathbb{Z}}$ is a Riesz system in $L_2(\mathbb{R})$. Denote by $\mathbb{S}_{B,\sigma}$ the space of functions~$s$ defined on $\mathbb{R}$ and representable in the form
\begin{equation}
	\label{1.4}
	s(x)=\sum\limits_{j\in\mathbb{Z}}\beta_jB\left(x-\frac{j\pi}{\sigma}\right),\quad\beta\in\ell_2(\mathbb{Z}).
\end{equation}

\begin{rem}
	The definition of Riesz system implies that for every function~$s$ of the form
	$$
	s(x)=\sum\limits_{j\in\mathbb{Z}}\beta_jB\left(x-\frac{j\pi}{\sigma}\right)
	$$
	the inclusions $s\in L_2(\mathbb{R})$ and $\beta\in\ell_2(\mathbb{Z})$ are equivalent, so we get that $\mathbb{S}_{B,\sigma}$ is a subspace of $L_2(\mathbb{R})$.
\end{rem}	

\begin{rem}
	It also follows from the definition of Riesz system that the functions~$D_{B,\sigma}$~and~$1/D_{B,\sigma}$ belong to $L_\infty[\,-\sigma,\sigma]$~\cite[Theorem~1.1.6]{skopina}.
\end{rem}

For $0<\rho<\sigma$, let $\mathbb{S}_{B,\sigma,\rho}$ be the space of functions~$s$ in $\mathbb{S}_{B,\sigma}$ that can be represented as~\eqref{1.4} with the additional condition 
$$
\sum\limits_{j\in\mathbb{Z}}\beta_je^{-i\tfrac{j\pi}{\sigma}y}=0\quad
\text{for a.e.}\quad \rho<|y|\leqslant\sigma
$$
(convergence of the series is interpreted in $L_2[\,-\sigma,\sigma]$).
When $\rho=\sigma$ by $\mathbb{S}_{B,\sigma,\sigma}$ we mean $\mathbb{S}_{B,\sigma}$.	

\begin{rem}
	For every $0<\rho\leqslant\sigma$, the subspaces $\mathbb{S}_{B,\sigma,\rho}$ are closed in $L_2(\mathbb{R})$ (see, for example,~\cite[Theorem~1.1.2]{skopina}).
\end{rem}

Let $s\in\mathbb{S}_{B,\sigma,\rho}$ be the function of the form~\eqref{1.4}. Put
$$
\zeta_{B,\sigma}(s,y)=\sum\limits_{j\in\mathbb{Z}}\beta_je^{-i\tfrac{j\pi}{\sigma}y}.
$$
Clearly, $\zeta_{B,\sigma}(s)\in L_2[\,-\sigma,\sigma]$. Furthermore, $\zeta_{B,\sigma}(s,y)=0$ for a.e.~$\rho<|y|\leqslant\sigma$. This implies, by Parseval's equality for product of two functions, the following representation of~$s$:
\begin{gather}
	\label{1.5}	s(x)=\int\limits_{-\rho}^{\rho}\zeta_{B,\sigma}(s,y)
	\Phi_{B,\sigma}(x,y)\,dy.
\end{gather}
Conversely, if the function~$s$ is defined by
$$
s(x)=\int\limits_{-\rho}^{\rho}\zeta(y)
\Phi_{B,\sigma}(x,y)\,dy
$$
for some $\zeta\in L_2[-\rho,\rho]$, then $s\in\mathbb{S}_{B,\sigma,\rho}$ and it admits representation~\eqref{1.4} with coefficients
$$
\beta_j=\frac{1}{2\sigma}\int\limits_{-\rho}^{\rho}\zeta(y)e^{i\tfrac{j\pi}{\sigma}y}\,dy,\quad j\in\mathbb{Z}.
$$

If $B$ is the $B$-spline
$$
B_{\sigma,m}(x)=\int\limits_\mathbb{R}\left(\frac{e^{i\frac{\pi}{\sigma}y}-1}
{i\frac{\pi}{\sigma}y}\right)^{m+1}e^{ixy}\,dy,
$$
we find that $\mathbb{S}_{B,\sigma}$ coincides with the space of splines $\mathbf{S}_{\sigma,\mu}$. The corresponding functions $\Phi_{B,\sigma}$ were introduced by Schoenberg (see~\cite{schnb}) and are called \emph{exponential splines}.

In general situation under consideration formula~\eqref{1.5} is a continual analog of decomposition over orthogonal bases in finite dimensional spaces of shifts of periodic functions (see,~\cite{we,me}) 

Let $\mathbb{S}_{B,\sigma,\rho}^1$ be the space of functions $s\in\mathbb{S}_{B,\sigma,\rho}$ for which in representation~\eqref{1.4} we have $\beta\in\ell_1(\mathbb{Z})$. Obviously, $\mathbb{S}_{B,\sigma,\rho}^1$ is dense in $\mathbb{S}_{B,\sigma,\rho}$.

\begin{lem}
	If $s\in\mathbb{S}_{B,\sigma,\rho}^1$, then
	\begin{equation}
		\label{1.6}
		\zeta_{B,\sigma}(s,y)=\frac{1}{2\pi D_{B,\sigma}(y)}\int\limits_{\mathbb{R}}
		s(t)\overline{\Phi_{B,\sigma}(t,y)}\,dt,
	\end{equation}
	 where convergence of the integral on the right-hand side is interpreted in  $L_2[\,-\sigma,\sigma]$.
\end{lem}

\begin{proof}
	First we note that for every $a,b\in\mathbb{R}$, $a<b$, the following inequality holds:
	$$
	\left(\int\limits_{-\sigma}^{\sigma}\left|\int\limits_a^b
	B(t)\overline{\Phi_{B,\sigma}(t,y)}\,dt
	-2\pi D_{B,\sigma}(y)\right|^2dy\right)^{1/2}\leqslant
	\sqrt{2\pi}\|D_{B,\sigma}\|^{1/2}_{L_{\infty}[\,-\sigma,\sigma]}
	\|B\|_{L_2(\mathbb{R})}.
	$$	
	Indeed, similarly to the proof of the proposition~$\Phi$4, we obtain the estimate
	\begin{gather*}
		\left(\int\limits_{-\sigma}^{\sigma}\left|\int\limits_a^b
		B(t)\overline{\Phi_{B,\sigma}(t,y)}\,dt
		-2\pi D_{B,\sigma}(y)\right|^2dy\right)^{1/2}\leqslant\\
		\leqslant2\pi\|D_{B,\sigma}\|^{1/2}_{L_{\infty}[\,-\sigma,\sigma]}
		\left(\int\limits_{\mathbb{R}}
		\left|\frac{1}{2\pi}\int\limits_a^b B(t)e^{-iyt}\,dt-\widehat{B}(y)\right|^2dy\right)^{1/2},
	\end{gather*}
	which implies the required inequality by Parseval's identity.
	
	For $n\in\mathbb{N}$, by formulas~\eqref{1.4}, \eqref{1.3}, property~$\Phi$2, and Minkowski inequality, we get
	\begin{gather*}
		\left(\int\limits_{-\sigma}^\sigma\left|
		\int\limits_{-n}^ns(t)\overline{\Phi_{B,\sigma}(t,y)}\,dt-
		2\pi D_{B,\sigma}(y)\zeta_{B,\sigma}(s,y)\right|^2dy\right)^{1/2}=\\
		=\left(\int\limits_{-\sigma}^\sigma\left|
		\sum\limits_{j\in\mathbb{Z}}\beta_j\int\limits_{-n}^n
		B\left(t-\frac{j\pi}{\sigma}\right)\overline{\Phi_{B,\sigma}(t,y)}\,dt
		-2\pi D_{B,\sigma}(y)\zeta_{B,\sigma}(s,y)\right|^2dy\right)^{1/2}=\\
		=\left(\int\limits_{-\sigma}^\sigma\left|
		\sum\limits_{j\in\mathbb{Z}}\beta_j\int\limits_{-n-\frac{j\pi}{\sigma}}^
		{n-\frac{j\pi}{\sigma}}
		B(t)\overline{\Phi_{B,\sigma}\left(t+\frac{j\pi}{\sigma},y\right)}\,dt
		-2\pi D_{B,\sigma}(y)\zeta_{B,\sigma}(s,y)\right|^2dy\right)^{1/2}=\\
		=\left(\int\limits_{-\sigma}^\sigma\left|
		\sum\limits_{j\in\mathbb{Z}}\beta_je^{-i\tfrac{j\pi}{\sigma}y}\left(
		\int\limits_{-n-\frac{j\pi}{\sigma}}^{n-\frac{j\pi}{\sigma}}
		B(t)\overline{\Phi_{B,\sigma}(t,y)}\,dt
		-2\pi D_{B,\sigma}(y)\right)\right|^2dy\right)^{1/2}\leqslant\\
		\leqslant\sum\limits_{j\in\mathbb{Z}}|\beta_j|
		\left(\int\limits_{-\sigma}^\sigma\left|\int\limits_{-n-\frac{j\pi}{\sigma}}^{n-\frac{j\pi}{\sigma}}
		B(t)\overline{\Phi_{B,\sigma}(t,y)}\,dt
		-2\pi D_{B,\sigma}(y)\right|^2dy\right)^{1/2}\xrightarrow[n\to\infty]{}0.
	\end{gather*}
	Termwise passing to the limit is valid, because for every $j\in\mathbb{Z}$ the expression in brackets tends to $0$ by property~$\Phi$4, while summability of $\beta$ and the remark at the beginning of the proof provide the estimate
	\begin{gather*}
		\sum\limits_{j\in\mathbb{Z}}|\beta_j|
		\left(\int\limits_{-\sigma}^\sigma\left|\int\limits_{-n-\frac{j\pi}{\sigma}}^{n-\frac{j\pi}{\sigma}}
		B(t)\overline{\Phi_{B,\sigma}(t,y)}\,dt
		-2\pi D_{B,\sigma}(y)\right|^2dy\right)^{1/2}\leqslant\\
		\leqslant \sqrt{2\pi}\|D_{B,\sigma}\|^{1/2}_{L_{\infty}[\,-\sigma,\sigma]}
		\|B\|_{L_2(\mathbb{R})}\sum\limits_{j\in\mathbb{Z}}|\beta_j|<\infty.
	\end{gather*}
	
	Term-by-term integration over a finite interval is valid since for every $n\in\mathbb{N}$ we have
	\begin{gather*}
		\int\limits_{-n}^n\sum\limits_{j\in\mathbb{Z}}|\beta_j|\left|B\left(t-\frac{j\pi}{\sigma}\right)\right|\left|\Phi_{B,\sigma}(t,y)\right|\,dt=
		\sum\limits_{j\in\mathbb{Z}}|\beta_j|\int\limits_{-n}^n\left|B\left(t-\frac{j\pi}{\sigma}\right)\right|\left|\Phi_{B,\sigma}(t,y)\right|\,dt\leqslant\\
		\leqslant\sum\limits_{j\in\mathbb{Z}}|\beta_j|
		\left(\int\limits_{-n}^n\left|B\left(t-\frac{j\pi}{\sigma}\right)\right|^2dt\right)^{1/2}\left(\int\limits_{-n}^n\left|\Phi_{B,\sigma}(t,y)\right|^2dt\right)^{1/2}\leqslant\\
		\leqslant
		\|B\|_{L_2(\mathbb{R})}\left(\int\limits_{-n}^n\left|\Phi_{B,\sigma}(t,y)\right|^2dt\right)^{1/2}\sum\limits_{j\in\mathbb{Z}}|\beta_j|<\infty.
	\end{gather*}
\end{proof}	

The following theorem is an analog of the Plancherel theorem for elements of spaces of shifts.

\begin{thm}
	If $S,s\in\mathbb{S}_{B,\sigma,\rho}$, then
	\begin{align}
		\label{1.7}
		\int\limits_{\mathbb{R}}|s(x)|^2dx&=
		2\pi\int\limits_{-\rho}^{\rho}\left|\zeta_{B,\sigma}(s,y)\right|^2
		D_{B,\sigma}(y)\,dy,\\
		\label{1.8}
		\int\limits_{\mathbb{R}}S(x)\overline{s(x)}\,dx&=
		2\pi\int\limits_{-\rho}^{\rho}\zeta_{B,\sigma}(S,y)
		\overline{\zeta_{B,\sigma}(s,y)}D_{B,\sigma}(y)\,dy.
	\end{align}
\end{thm}

\begin{proof}
	Let us prove~\eqref{1.7}. It suffices to check the equality for $s\in\mathbb{S}_{B,\sigma,\rho}^1$; it will hold for $s\in\mathbb{S}_{B,\sigma,\rho}$ by continuity.
	
	Using~\eqref{1.5}~and~\eqref{1.6}, we derive to
	\begin{gather*}
		\int\limits_{\mathbb{R}}|s(x)|^2dx=
		\int\limits_{\mathbb{R}}s(x)\overline{s(x)}\,dx=\lim\limits_{n\to\infty}
		\int\limits_{-n}^n\overline{s(x)}\int\limits_{-\rho}^{\rho}
		\zeta_{B,\sigma}(s,y)\Phi_{B,\sigma}(x,y)\,dy\,dx=\\
		=\lim\limits_{n\to\infty}\int\limits_{-\rho}^{\rho}
		\zeta_{B,\sigma}(s,y)\int\limits_{-n}^n\overline{s(x)}
		\Phi_{B,\sigma}(x,y)\,dx\,dy=\\
		=\lim\limits_{n\to\infty}\int\limits_{-\rho}^{\rho}
		\zeta_{B,\sigma}(s,y)\overline{\int\limits_{-n}^ns(x)
			\overline{\Phi_{B,\sigma}(x,y)}\,dx}\,dy
		=2\pi\int\limits_{-\rho}^{\rho}\left|\zeta_{B,\sigma}(s,y)\right|^2
		D_{B,\sigma}(y)\,dy.
	\end{gather*}
	Passage to the limit under the integral sign is valid by Lemma~1.
	
	Interchanging the order of integration is justified since for any $n\in\mathbb{N}$ the function
	$$
	(x,y)\mapsto \overline{s(x)}\zeta_{B,\sigma}(s,y)\Phi_{B,\sigma}(x,y)
	$$ 
	is integrable on $[\,-n,n]\times[\,-\sigma,\sigma]$. Indeed,
	\begin{gather*}
		\int\limits_{-n}^n|s(x)|\int\limits_{-\sigma}^{\sigma}
		|\zeta_{B,\sigma}(s,y)|\left|\Phi_{B,\sigma}(x,y)\right|\,dy\,dx
		\leqslant\\
		\leqslant\left(\int\limits_{-\sigma}^{\sigma}|\zeta_{B,\sigma}(s,y)|^2
		dy\right)^{1/2}\int\limits_{-n}^n|s(x)|
		\left(\int\limits_{-\sigma}^{\sigma}\left|\Phi_{B,\sigma}(x,y)\right|^2dy\right)^{1/2}dx\leqslant\\
		\leqslant\left(\int\limits_{-\sigma}^{\sigma}|\zeta_{B,\sigma}(s,y)|^2
		dy\right)^{1/2}
		\left(\int\limits_{-n}^n|s(x)|^2dx\right)^{1/2}
		\left(\int\limits_{-n}^n\int\limits_{-\sigma}^{\sigma}
		|\Phi_{B,\sigma}(x,y)|^2
		dy\,dx\right)^{1/2}<\infty.
	\end{gather*}
	
	Formula~\eqref{1.8} is derived from~\eqref{1.7} in a standard way.
\end{proof}

The Fourier transform and the partial Fourier integral of the function~$f\in L_2(\mathbb{R})$ with respect to the system~$\Phi_{B,\sigma}$ are defined as follows:
\begin{align}
	\label{1.9}
	\zeta_{B,\sigma}(f,y)&=\frac{1}{2\pi D_{B,\sigma}(y)}\int\limits_{\mathbb{R}}
	f(t)\overline{\Phi_{B,\sigma}(t,y)}\,dt,\\
	\label{1.10}
	J_{B,\sigma,\rho}(f,x)&=
	\int\limits_{-\rho}^{\rho}\zeta_{B,\sigma}(f,y)\Phi_{B,\sigma}(x,y)\,dy,
	\quad 0<\rho\leqslant\sigma.
\end{align}
Let us show that the integral on the right-hand side of~\eqref{1.9} converges in $L_2[\,-\sigma,\sigma]$. It suffices to prove that the sequence
\begin{equation}\label{1.1010}
	\left\lbrace y\mapsto\int\limits_{-n}^n
	f(t)\overline{\Phi_{B,\sigma}(t,y)}\,dt\right\rbrace_{n\in\mathbb{N}} 	
\end{equation}
is a Cauchy sequence in $L_2[\,-\sigma,\sigma]$. Take $\varepsilon>0$. Since the sequence
$$
\left\lbrace y\mapsto\int\limits_{-n}^n
f(t)e^{-iyt}\,dt\right\rbrace_{n\in\mathbb{N}} 
$$
is convergent in $L_2(\mathbb{R})$, there exists such number $N\in\mathbb{N}$ that for every $n,k\in\mathbb{N}$, $n>k>N$, the following inequality holds:
$$
\int\limits_{\mathbb{R}}\left|\;\int\limits_{[\,-n,n]\setminus[\,-k,k]}
f(t)e^{-iyt}\,dt\right|^2dy<\frac{\varepsilon}{\|D_{B,\sigma}\|_{L_{\infty}[\,-\sigma,\sigma]}}.
$$

Therefore, for such $n$ and $k$, we have the estimate
\begin{gather*}
	\int\limits_{-\sigma}^\sigma\left|\int\limits_{-n}^n
	f(t)\overline{\Phi_{B,\sigma}(t,y)}\,dt-\int\limits_{-k}^k
	f(t)\overline{\Phi_{B,\sigma}(t,y)}\,dt\right|^2dy=\\
	=\int\limits_{-\sigma}^\sigma\left|\;\int\limits_{[\,-n,n]
		\setminus[\,-k,k]}f(t)\overline{\Phi_{B,\sigma}(t,y)}\,dt\right|^2dy=\\
	=\int\limits_{-\sigma}^\sigma\left|\sum\limits_{\nu\in\mathbb{Z}}
	\overline{\widehat{B}(y+2\nu\sigma)}\int\limits_{[\,-n,n]
		\setminus[\,-k,k]}f(t)e^{-i(y+2\nu\sigma)t}\,dt\right|^2dy\leqslant\\
	\leqslant\|D_{B,\sigma}\|_{L_{\infty}[\,-\sigma,\sigma]}
	\int\limits_{-\sigma}^\sigma\left(\sum\limits_{\nu\in\mathbb{Z}}\left|\;\int\limits_{[\,-n,n]
		\setminus[\,-k,k]}f(t)e^{-i(y+2\nu\sigma)t}\,dt\right|^2\right)\,dy=\\
	=\|D_{B,\sigma}\|_{L_{\infty}[\,-\sigma,\sigma]}\int\limits_{\mathbb{R}}\left|\;\int\limits_{[\,-n,n]\setminus[\,-k,k]}
	f(t)e^{-iyt}\,dt\right|^2dy<\varepsilon.	
\end{gather*}	
This means that~\eqref{1.1010} is a Cauchy sequence in $L_2[\,-\sigma,\sigma]$.

Since $1/D_{B,\sigma}\in L_{\infty}[\,-\sigma,\sigma]$, we get $\zeta_{B,\sigma}(f)\in L_2[\,-\sigma,\sigma]$, which implies $J_{B,\sigma,\rho}(f)\in\mathbb{S}_{B,\sigma,\rho}$ and, in particular, $J_{B,\sigma,\rho}(f)\in L_2(\mathbb{R})$.

Furthermore, for $f\in L_2(\mathbb{R})$ and $s\in\mathbb{S}_{B,\sigma,\rho}$, by formulas~\eqref{1.5}~and~\eqref{1.9}, we obtain
\begin{gather*}
	\int\limits_{\mathbb{R}}f(t)\overline{s(t)}\,dt=\lim\limits_{n\to\infty}
	\int\limits_{-n}^n f(t)\overline{s(t)}\,dt=\lim\limits_{n\to\infty}
	\int\limits_{-n}^n f(t)\int\limits_{-\rho}^{\rho}\overline{\zeta_{B,\sigma}(s,y)}
	\overline{\Phi_{B,\sigma}(t,y)}\,dy\,dt=\\
	=\lim\limits_{n\to\infty}\int\limits_{-\rho}^{\rho}
	\overline{\zeta_{B,\sigma}(s,y)}\int\limits_{-n}^nf(t)\overline{\Phi_{B,\sigma}(t,y)}\,dt\,dy=\\
	=2\pi\int\limits_{-\rho}^{\rho}\overline{\zeta_{B,\sigma}(s,y)}
	\zeta_{B,\sigma}(f,y)D_{B,\sigma}(y)\,dy,
\end{gather*}
while the equality~\eqref{1.8} yields
\begin{gather*}
	\int\limits_{\mathbb{R}}J_{B,\sigma,\rho}(f,t)\overline{s(t)}\,dt=
	2\pi\int\limits_{-\rho}^{\rho}\overline{\zeta_{B,\sigma}(s,y)}
	\zeta_{B,\sigma}(f,y)D_{B,\sigma}(y)\,dy.
\end{gather*}

This implies that $f-J_{B,\sigma,\rho}(f)\perp\mathbb{S}_{B,\sigma,\rho}$,
i.e., that $J_{B,\sigma,\rho}(f)$ is an orthogonal projection of $f$ on $\mathbb{S}_{B,\sigma,\rho}$. In addition, 
\begin{gather}
	\label{1.11}
	2\pi\int\limits_{-\rho}^{\rho}\left|\zeta_{B,\sigma}(f,y)\right|^2
	D_{B,\sigma}(y)\,dy=
	\int\limits_{\mathbb{R}}\left|J_{B,\sigma,\rho}(f,x)\right|^2dx\leqslant
	\int\limits_{\mathbb{R}}|f(x)|^2\,dx.
\end{gather}

Now we provide an explicit expression of the new Fourier transform and best approximation by spaces of shifts in terms of trigonometric Fourier transform.

\begin{thm}
	If $f\in L_2(\mathbb{R})$, then
	\begin{gather}
		\label{1.12}
		\zeta_{B,\sigma}(f)=\frac{1}{D_{B,\sigma}}\sum\limits_{k\in\mathbb{Z}}
		\overline{\widehat{B}(\cdot+2k\sigma)}\widehat{f}(\cdot+2k\sigma)\;
		\text{ in }\;L_2[\,-\sigma,\sigma],\\
		\label{1.13}
		E^2(f,\mathbb{S}_{B,\sigma,\rho})_2=
		2\pi\left(\int\limits_{\mathbb{R}}\left|\widehat{f}(y)\right|^2dy-
		\int\limits_{-\rho}^{\rho}\frac{1}{D_{B,\sigma}(y)}
		\left|\sum\limits_{k\in\mathbb{Z}}
		\overline{\widehat{B}(y+2k\sigma)}\widehat{f}(y+2k\sigma)\right|^2dy\right).
	\end{gather}
\end{thm}
\begin{proof}
	By formulas~\eqref{1.9},~\eqref{1.2} and the Cauchy–Schwarz inequality, we have 
	\begin{gather*}	
		\int\limits_{-\sigma}^\sigma\left|\zeta_{B,\sigma}(f,y)-\frac{1}{D_{B,\sigma}(y)}\sum\limits_{k\in\mathbb{Z}}
		\overline{\widehat{B}(y+2k\sigma)}\widehat{f}(y+2k\sigma)\right|^2dy=\\
		=\lim\limits_{n\to\infty}\int\limits_{-\sigma}^\sigma\left|\frac{1}{D_{B,\sigma}(y)}\sum\limits_{k\in\mathbb{Z}}
		\overline{\widehat{B}(y+2k\sigma)}\left(\frac{1}{2\pi}\int\limits_{-n}^nf(t)e^{-i(y+2k\sigma)t}dt-\widehat{f}(y+2k\sigma)\right)\right|^2dy\leqslant\\
		\leqslant\lim\limits_{n\to\infty}\int\limits_{-\sigma}^\sigma
		\sum\limits_{k\in\mathbb{Z}}\left|\frac{1}{2\pi}\int\limits_{-n}^nf(t)e^{-i(y+2k\sigma)t}dt-\widehat{f}(y+2k\sigma)\right|^2dy=\\
		=\lim\limits_{n\to\infty}\int\limits_{\mathbb{R}}\left|\frac{1}{2\pi}\int\limits_{-n}^nf(t)e^{-iyt}dt-\widehat{f}(y)\right|^2dy=0.
	\end{gather*}	
	
	Using properties of orthogonal projection, we get
	$$
	E^2(f,\mathbb{S}_{B,\sigma,\rho})_2=\|f-J_{B,\sigma,\rho}f\|^2_{L_2(\mathbb{R})}=
	\|f\|^2_{L_2(\mathbb{R})}-\|J_{B,\sigma,\rho}f\|^2_{L_2(\mathbb{R})}.
	$$
	This implies, by the Plancherel theorem and formulas~\eqref{1.11} and~\eqref{1.12}, the following equality:
	\begin{gather*}
		E^2(f,\mathbb{S}_{B,\sigma,\rho})_2=2\pi\int\limits_{\mathbb{R}}\left|\widehat{f}(y)\right|^2dy-
		2\pi\int\limits_{-\rho}^{\rho}\left|\zeta_{B,\sigma}(f,y)\right|^2	D_{B,\sigma}(y)\,dy=\\
		=2\pi\left(\int\limits_{\mathbb{R}}\left|\widehat{f}(y)\right|^2dy-
		\int\limits_{-\rho}^{\rho}\frac{1}{D_{B,\sigma}(y)}
		\left|\sum\limits_{k\in\mathbb{Z}}
		\overline{\widehat{B}(y+2k\sigma)}\widehat{f}(y+2k\sigma)\right|^2dy\right).
	\end{gather*}
\end{proof}

Note that particular case of Fourier analysis in spline spaces was developed by Vinogradov~\cite{ol5}.


The results obtained have several applications. First of all, we intend to use formula~\eqref{1.13} to derive a criteria of optimality of spaces $\mathbb{S}_{B,\sigma}$ for various classes of functions in $L_2(\mathbb{R})$. Secondly, we expect to apply the results for counting the average dimension of the spaces $\mathbb{S}_{B,\sigma,\rho}$.

\subsection*{Funding}
The reported study was funded by RFBR according to the research project № 20-31-90025.

\subsection*{Competing interests}
The author has no competing interests to declare.

\subsection*{Data availability}
The manuscript has no associated data.

\end{document}